\documentclass[reprint,superscriptaddress,nofootinbib,
aps,pra
]{revtex4-2}

\usepackage{amsmath,amssymb,amsxtra,amsthm,hyperref,mathtools}
\usepackage[normalem]{ulem}
\usepackage{braket,appendix}
\usepackage{graphicx}
\usepackage{subcaption}
\usepackage{caption}
\usepackage{color}
\hypersetup{
     colorlinks = true,
     linkcolor = magenta,
     anchorcolor = black,
     citecolor = blue,
     filecolor = blue,
     urlcolor = blue
     }

\usepackage{pifont}
\newcommand{\cmark}{\ding{51}}%
\newcommand{\xmark}{\ding{55}}%
\newcommand{\mycomment}[1]{}

\usepackage{enumitem}

\newcommand{\Bmath}[1]{\mbox{\bf {#1}}}

\def\v{{\Bmath v}}

\newtheorem{theorem}{Theorem}[section]

\newtheorem{corollary}[theorem]{Corollary}

\newtheorem{lemma}[theorem]{Lemma}

\theoremstyle{definition}

\usepackage{xcolor}

\begin{document}

\title{Efficient Digital Quadratic Unconstrained Binary Optimization Solvers for SAT Problems}

\author{Robert Simon Fong}
\email{r.s.fong@bham.ac.uk}
\affiliation{School of Computer Science, University of Birmingham, Birmingham, B15 2TT, UK}

\author{Yanming Song}
\email{yanmings@connect.hku.hk}
\affiliation{Department of Computer Science, The University of Hong Kong, Pokfulam Road, Hong Kong SAR}

\author{Alexander Yosifov}
\email{alexanderyyosifov@gmail.com}
\affiliation{Space Research and Technology Institute, Bulgarian Academy of Sciences, Sofia, 1113, Bulgaria}

\date{\today}

\begin{abstract}
Boolean satisfiability (SAT) is a propositional logic problem of determining whether an assignment of variables satisfies a Boolean formula. Many combinatorial optimization problems can be formulated in Boolean SAT logic --- either as $k$-SAT decision problems or Max $k$-SAT optimization problems, with \textit{conflict-driven} (CDCL) solvers being the most prominent. Despite their ability to handle large instances, CDCL-based solvers have fundamental scalability limitations. In light of this, we propose recently-developed quadratic unconstrained binary optimization (QUBO) solvers as an alternative platform for $3$-SAT problems. To utilize them, we implement a $2$-step [$3$-SAT]-[Max $2$-SAT]-[QUBO] conversion procedure and present a rigorous proof to calculate the number of both satisfied and violated clauses of the original $3$-SAT instance from the transformed Max $2$-SAT formulation. We then demonstrate, through numerical simulations on several benchmark instances, that \textit{digital} QUBO solvers can achieve state-of-the-art accuracy on $78$-variable $3$-SAT benchmark problems. Our work facilitates the broader use of quantum annealers on noisy intermediate-scale quantum (NISQ) devices, as well as their quantum-inspired digital counterparts, for solving $3$-SAT problems. 
\end{abstract}

\maketitle

\section{Introduction}
The Boolean $k$-SAT problem (for $k\geq 3$) is an NP-complete \cite{SCOOK71} combinatorial optimization problem with many commercial applications \cite{10.1145/337292.337611, 10.1145/296399.296450, 10.1145/1146909.1147034, 10.1109/ICCAD51958.2021.9643505, 108614, 1097859}, making it an invaluable tool for addressing complex decision-making and problem-solving tasks. SAT problems are solved by determining whether an assignment of Boolean variables, either \texttt{TRUE} or \texttt{FALSE}, satisfies a propositional Boolean formula in conjunctive normal form (CNF), such that the formula evaluates to \texttt{TRUE}. Fundamentally, this is characterized by search for an \textit{optimal} configuration (\textit{e.g.} minimum of an objective function) of variables within a vast but finite solution space.

Consequently, most SAT solvers utilize search algorithms in conjunction with decision and branching heuristics \cite{sa2, sa3}. Where the CDCL framework \cite{CMV21}, due to its proven efficiency in solving industrial SAT problems \cite{Authenticus:P-009-199}, has emerged as the backbone of modern solvers. Typically, high-performance CDCL solvers employ a multitude of sophisticated clause minimization \cite{a15090302}, distance \cite{distance}, and history-based \cite{Liang_Ganesh_Poupart_Czarnecki_2016} heuristics, tailored for crafted benchmarks. Despite accurately solving larger instances, CDCL solvers have inherent scaling constraints \cite{scaling}. While some \textit{learning-focused} search techniques exist, attempts to parallelize CDCL solvers often lead to excessive memory usage, which degrades the performance and can render them unsuitable for some inherently large-scale or computationally-hard problems \cite{7225110, 4397270}.

The ubiquitous nature of SAT problems prompted the search for quantum computing solutions, where recent hardware advancements \cite{dwave123} enabled the embedding of the QUBO form of many discrete optimization problems (\textit{e.g.} SAT) \cite{NPmapping} to quantum annealing devices (see, \textit{e.g.}, \cite{Date2019EfficientlyEQ}), such that the objective becomes minimizing a quadratic polynomial over binary variables.\footnote{By virtue of the $3$-SAT problem being NP-complete, any solver that can efficiently solve it can be also used for all problems in the NP class with only polynomial time to verify solutions. Moreover, every SAT problem can be expressed as a 3-SAT problem in CNF \cite{reduce}.} The hope is that, due to their noise resilience (compared to some gate-based models), they can be applied on problems with non-trivial logic structure in the NISQ era. Still, despite continued efforts, demonstrating a clear scaling or accuracy advantage remains challenging. 

Regardless, the main underlying principles of adiabatic quantum computing remain valid, thus stimulating the development of \textit{digital} alternatives; so-called quantum-inspired QUBO solvers. Indeed, they recently emerged as a viable alternative due to several desirable features, such as parameter tuning flexibility and inherent capacity to effectively handle large instances \textit{in parallel}, see \cite{nature}. To that end, progress in ground state convergence \cite{PhysRevE.58.5355}, energy landscape exploration \cite{energy}, and multidimensional function minimization \cite{minimization} bolster the ability of QUBO solvers to successfully tackle discrete optimization problems.

In the absence of full-fledged quantum computers, QUBO solvers can represent a prominent general-purpose optimization platform. The present work is a step in this direction as we propose a $2$-step [$3$-SAT]-[Max $2$-SAT]-[QUBO] conversion procedure for embedding $3$-SAT problems of arbitrary size to available QUBO solvers (see Section \ref{sec:sec2} and Appendix \ref{appendix} and \ref{appendixb} for details). By construction, the conversion allows us to retrieve the solution of the original $3$-SAT problem from the QUBO solution. Then, in Section \ref{sec:properties}, we present a rigorous proof to retrieve the number of both \textit{satisfied} and \textit{violated} clauses of the original 3-SAT instance from the transformed Max $2$-SAT formulation. Finally, in Section \ref{sec:experiment}, we conduct numerical simulations on (i) several publicly available benchmark instances of varying sizes and clause densities, and (ii) randomly-generated instances in the \textit{hard} region, as defined in the seminal works \cite{mezard, mezard2}. In both scenarios we show that QUBO solvers can indeed match the accuracy of RC2 \cite{rc2} --- a CDCL-based state-of-the-art Max SAT solver. To the best of our knowledge, this is the first direct performance-oriented comparison of commercial QUBO solvers to RC$2$, and fills an important gap in the literature. Our results offer theoretical advancements in the field of discrete optimization and establish the QUBO platform as a high-performance alternative to CDCL for solving $3$-SAT problems on near-term devices.

\section{Converting 3-SAT problems to QUBO}
\label{sec:sec2}
To use available digital QUBO solvers and take advantage of their inherent properties, such as parameter tuning flexibility and scalability, the logical structure of the original $3$-SAT problem must be converted to the appropriate form, see \cite{NPmapping}. Here, we implement a $2$-step [$3$-SAT]-[Max $2$-SAT]-[QUBO] conversion with the so-called $(7,10)$-gadget \cite{GAREY1976237}. For that reason, we herein define ``QUBO" as any heuristic that can efficiently solve hard $3$-SAT problems, whose use is enabled by the $(7,10)$-gadget described in \cite{GAREY1976237} and Section \ref{sec:3to2} below. Moreover, the presented method is necessary since the $3$-SAT Hamiltonian does not have an Ising quantization, see Appendix \ref{appendix} for a detailed proof. 

\subsection{3-SAT to Max 2-SAT}
\label{sec:3to2}

Consider an arbitrary $3$-SAT problem $\dagger$ with $M$ clauses and $N$ literals expressed in CNF, that is a set of logical \texttt{AND} operations on a string of $M$ disjunctive $3$-literal clauses $\left\{\Omega_k\right\}_{k= 1}^M$ each expressed as logical \texttt{OR} operations on its $N$ literals $\left\{\sigma_i\right\}_{i=1}^N$. Then we can write $\dagger = \Omega_{1} \land \Omega_{2} \land \dots \land \Omega_{M}$ with the $k^{th}$ clause expressed as:
\begin{align*}
    \Omega_{k}=\left(\operatorname{sgn}\left(\sigma^{k}_{j_{1}}\right)\sigma^{k}_{j_{1}}\lor\operatorname{sgn}\left(\sigma^{k}_{j_{2}}\right)\sigma^{k}_{j_{2}}\lor\operatorname{sgn}\left(\sigma^{k}_{j_{3}}\right)\sigma^{k}_{j_{3}}\right), 
\end{align*}
 where $k\in \{1,2,\dots,M\}$, $j_1, j_2, j_3 \in \left\{1,\ldots,N\right\}$, and $\operatorname{sgn}\left(\sigma^{k}_{j_{i}}\right)\sigma^{k}_{j_{i}} = \sigma^{k}_{j_{i}} \text{ or } \neg \sigma^{k}_{j_{i}}$, where equivalently, $\operatorname{sgn}\left(\sigma^{k}_{j_{i}}\right) = \mathtt{TRUE}$ and $\mathtt{FALSE}$, respectively. For the rest of the paper we will, without loss of generality, assume $\dagger$ consists of clauses with exactly $3$ literals.
 
Without loss of generality, we now illustrate how the \textbf{$(7,10)$-gadget} \cite{GAREY1976237} converts a $3$-SAT problem to Max $2$-SAT with a simple clause $\Omega_k =\left(\sigma^{k}_{j_{1}}\lor\sigma^{k}_{j_{2}}\lor\sigma^{k}_{j_{3}}\right)$, where the construction extends naturally to clauses with negations of literals. First, we define exactly \textbf{one} ancillary variable $d^k$, and then replace $\Omega_k$ by the corresponding set of $10$ Max $2$-SAT clauses:

\begin{align}
\label{eqn:Max2from3}
    \Omega'_k := & \left\{ (\sigma^{k}_{j_{1}}), (\sigma^{k}_{j_{2}}), (\sigma^{k}_{j_{3}}), (d^{k}), \right. \nonumber \\
    & (\neg \sigma^{k}_{j_{1}} \lor \neg\sigma^{k}_{j_{2}}), (\neg\sigma^{k}_{j_{1}} \lor \neg\sigma^{k}_{j_{3}}), (\neg\sigma^{k}_{j_{2}} \lor \neg\sigma^{k}_{j_{3}}), \nonumber \\
    & \left.({\sigma}^{k}_{j_{1}} \lor \neg d^{k}), ({\sigma}^{k}_{j_{2}} \lor \neg d^{k}), ({\sigma}^{k}_{j_{3}} \lor \neg d^{k}) \right\}.
\end{align}
After replacing every clause of the original $3$-SAT problem in this manner, the conversion of the of the entire $3$-SAT problem can therefore be expressed compactly as:

\begin{align*}
    \text{3-SAT} &\rightarrow \text{Max 2-SAT} \\
    \bigwedge_{k=1}^M \Omega_k &\mapsto \bigwedge_{k=1}^{M} \Omega'_k.
\end{align*}
While the index of the converted Max $2$-SAT problem goes to $M$, each $\Omega_k'$ contains $10$ clauses.
Therefore, an arbitrary $3$-SAT problem with $M$ clauses can be converted to a Max $2$-SAT problem with $10M$ clauses and $N+M$ variables, see Table \ref{table:Max2from3:config}.

Evidently, there are $8$ possible choices of $(\sigma^{k}_{j_{1}}, \sigma^{k}_{j_{2}}, \sigma^{k}_{j_{3}})$, with each being either $0$ or $1$. Amongst them, only $(0,0,0)$ will not satisfy $\Omega_k$. In particular, given that an assignment $(\sigma^{k}_{j_{1}}, \sigma^{k}_{j_{2}}, \sigma^{k}_{j_{3}})$ satisfies $\Omega_k$, then exactly $7$ clauses of $\Omega'_k$ will be satisfied. Otherwise, if the initial assignment does not satisfy $\Omega_k$, then exactly $6$ clauses of $\Omega'_k$ will be satisfied. This is summarized in Table \ref{table.sat_vio_table} in Section \ref{sec:properties}, where the properties of the conversion are studied.

\begin{table}[ht!]
\begin{center}
\begin{tabular}{|l|l|l|}
\hline
         & 3-SAT              & Max 2-SAT             \\ \hline
Clauses  & $M$  & $ 10M$ \\ \hline
Literals & $N$ & $N +M$ \\ \hline
\end{tabular}
\caption{Changes in total number of clauses and literals in the $3$-SAT to Max $2$-SAT conversion (Eq. \eqref{eqn:Max2from3}).}
\label{table:Max2from3:config}
\end{center}
\end{table}

\subsection{Max 2-SAT to QUBO}
\label{sec:2toqubo}
The previous Section outlined the conversion of an arbitrary $3$-SAT problem to a Max $2$-SAT problem. Here, we present the second step of the procedure. Consider a Max 2-SAT problem $\ddagger := \bigwedge_{k=1}^{M'} \Omega'_k$ with $M'$ clauses $\left\{ \Omega'_k\right\}_{k=1}^{M'} $and $N'$ literals $\left\{\sigma'_i\right\}_{i=1}^{N'}$, where each clause contains at most two literals. We now define a QUBO problem on $\ddagger$ over the search space $\mathbf{x} \in \left\{0,1\right\}^{N'}$, with the goal to minimize the total number of violated clauses. Formally:

We first define a set of variables $\{v_{j}^{k}\}$ over the $M'$ clauses. Here, $v_{j}^{k}\in \{-1,0,1\}$, where $v_{j}^{k}=-1 (+1)$ if $\sigma'_{j}$ is negated (unnegated) in the $k^{th}$ clause, respectively, and $v_{j}^{k}=0$ for all clauses where $\sigma'_{j}$ does not appear in the clause. More precisely, for each $\Omega'_k \in \ddagger$ containing literals with sub-indices $j_1,j_2$, we define variables $\left\{v_j^k\right\}$ where:
\begin{enumerate}
    \item If $j = j_1$ or $j_2$:
    \begin{align*}
v_{j}^k = 2 \cdot \operatorname{sgn}\left(\sigma'^{k}_{j}\right) - 1 \Leftrightarrow \begin{cases} 
      1 & \operatorname{sgn}\left(\sigma'^{k}_{j}\right) = \texttt{TRUE} \\
      -1 & \operatorname{sgn}\left(\sigma'^{k}_{j}\right) = \texttt{FALSE}
   \end{cases}
\end{align*}
    \item $v_j^k = 0$ otherwise.
\end{enumerate}

For each $k \in \{1,\ldots,M'\}$, we can isolate the entries $x_{j_i}^k$ from the input $\mathbf{x} = \left(x_1,\ldots,x_{j_1}^k, \ldots, x_{j_2}^k,\ldots, x_{N'}\right) \in \left\{0,1\right\}^{N'}$, and  define the following intermediate variable $S_{j_i}^k$:
\begin{align}
\label{eq:intermediate2}
S_{j_i}^k &=\left( \frac{1+v_{j_i}^k}{2}-v_{j_i}^k x_{j_i}^k \right). \nonumber \\ 
S^k &=S_{j_1}^k\cdot S_{j_2}^k = \prod_{i=1}^{2}\left( \frac{1+v_{j_i}^k}{2}-v_{j_i}^k x_{j_i}^k \right).
\end{align}

One can check that if $\mathbf{x}$ satisfies $\Omega'_k$, then $S^k = 0$ and $S^k = 1$ otherwise. The total number of violated clauses in Max $2$-SAT problem $\ddagger = \bigwedge_{k=1}^{M'} \Omega'_k$, given input $\mathbf{x}$, is:
\begin{align*}
\mathcal{V}(\mathbf{x}) & =\sum_{k=1}^{M'}\prod_{i=1}^{2}\left( \frac{1+v_{j_i}^k}{2}-v_{j_i}^k x_{j_i}^k \right) \\
& =\sum_{k=1}^{M'}\Bigg(v_{j_1}^k v_{j_2}^k x_{j_1}^k x_{j_2}^k - v_{j_1}^k\frac{1+v_{j_2}^k}{2}x_{j_1}^k \\
& \quad \quad - v_{j_2}^k\frac{1+v_{j_1}^k}{2}x_{j_2}^k  + \frac{1+v_{j_1}^k}{2}\frac{1+v_{j_2}^k}{2}\Bigg),
\end{align*}
where the optimal solution $\mathbf{x^*}$ is thus given by:
\begin{equation}
\mathbf{x}^*=\mathop{\arg\min}\limits_{\mathbf{x} \in \left\{0,1\right\}^{N'}} \mathcal{V}(\mathbf{x}).
\end{equation}

Equivalently, upon re-scaling $\mathcal{V}(\mathbf{x})$ by a factor of $2$ and dropping the constant term, we obtain the QUBO objective function $q(\mathbf{x})$ corresponding to $\mathcal{V}(\mathbf{x})$, which is to be minimized:
\begin{equation}
\label{eq:qubo}
q(\mathbf{x})=\frac{1}{2}\mathbf{x}^\top \mathbf{Q}\mathbf{x} + \mathbf{b}^\top \mathbf{x} + c, 
\end{equation}
where for all $j_1, j_2 \in \{1,\ldots,N'\}$:
\begin{equation}
Q_{j_1 j_2}=2\sum_{k=1}^{M'}v_{j_1}^k v_{j_2}^k
\end{equation}
\begin{equation}
b_{j_1}=-\sum_{k=1}^{M'}v_{j_1}^k\left( v_{j_2}^k+1 \right).
\end{equation}
For the unassigned $Q_{ij}$ and $b_i$ in the above procedure, we set $Q_{ij}=0$, $b_{i}=0$.

\section{Properties}
\label{sec:properties}

In this Section, we illustrate how information about the original $3$-SAT problem can be extracted from the converted Max $2$-SAT formulation. In particular, we show how to retrieve the number of both violated and satisfied clauses of the initial $3$-SAT problem.

We denote the initial $3$-SAT problem by $ \dagger := \bigwedge_{k=1}^M \Omega_k$ and $\ddagger := \bigwedge_{k=1}^{M} \Omega'_k$ the Max $2$-SAT converted from $\dagger$.
Recall from Section \ref{sec:3to2}, for each clause $\Omega_{k}=\left(\sigma^{k}_{j_{1}}\lor\sigma^{k}_{j_{2}}\lor\sigma^{k}_{j_{3}}\right)$ of the original $3$-SAT problem $\dagger$ there are $8$ possible choices of assignments of $\left(\sigma^{k}_{j_{1}},\sigma^{k}_{j_{2}},\sigma^{k}_{j_{3}}\right)$ with $\sigma_{j_i}^k \in \left\{0,1\right\}$ for $i = 1,2,3$. By Theorem 1.1 of \cite{GAREY1976237}, $\Omega_k$ can be converted to a Max $2$-SAT clause $\Omega_k'$, where the satisfactory table is outlined in Table \ref{table.sat_vio_table}.

\begin{table}[ht!]
\begin{center}
\begin{tabular}{|c|c|c|c|c|}
\hline
& $\left(\sigma^{k}_{j_{1}},\sigma^{k}_{j_{2}},\sigma^{k}_{j_{3}}\right)$ & $d^k$   & $\Omega_k'$ & $\Omega_k$ \\ \hline
Case 1 & $(0,0,0)$                                                              & $0$   & $6$ \cmark   & \xmark      \\ \cline{3-5} 
                &                                                       & $1$   & $4$ \cmark   & \xmark      \\ \hline
Case 2 & $(0,0,1), (0,1,0)$,                                                    & $0$   & $7$ \cmark   & \cmark      \\  \cline{3-5} 
               &         $(1,0,0)$                                      & $1$   & $6$ \cmark   & \cmark      \\ \hline
Case 3 & Otherwise                                                              & $0,1$ & $7$ \cmark   & \cmark      \\ \hline
\end{tabular}
\end{center}
\caption{Satisfactory Table of a single clause the $3$-SAT to Max $2$-SAT conversion described by Eq. \eqref{eqn:Max2from3}. For the column corresponding to $\Omega_k'$, the label ($\#$ \cmark) denote the number ($\#$) of satisfied clause(s). For the column corresponding to $\Omega_k$, \cmark and \xmark  denote satisfied and violation respectively.}
\label{table.sat_vio_table}
\end{table}
Hence, we arrive at the following statement, which completes the results in \cite{reduce2}
\begin{lemma}
Given a $3$-SAT clause $\Omega_k$ and its corresponding Max $2$-SAT conversion $\Omega_k'$, there exists a configuration of ancillary variable $d^k$ such that: if an assignment satisfies $\Omega_k$ then exactly $7$ clauses are satisfied in $\Omega_k'$, otherwise if an assignment does not satisfy $\Omega_k$ then exactly $6$ clauses are satisfied in $\Omega_k'$.
\end{lemma}

For the rest of the Section, we show that the number of violated and satisfied clauses of the original $3$-SAT problem, \textbf{for each of the cases} outlined in Table \ref{table.sat_vio_table}, can be retrieved from the converted Max $2$-SAT instance once the problem $\ddagger$ is solved. We begin with the special case where the converted Max $2$-SAT problem is solved \emph{and} all the ancillary variables are identically zero. 

\begin{theorem}
\label{thm:retrieve:d0}
Let $\ddagger := \bigwedge_{k=1}^{M} \Omega'_k$ denote a Max $2$-SAT problem converted from $3$-SAT problem $ \dagger := \bigwedge_{k=1}^M \Omega_k$. Suppose $\ddagger$ is solved when the ancillary variables all vanish and the final number of satisfied clauses of $\ddagger$ is known. Then the number of violated and satisfied clauses of $\dagger$ can be retrieved. 
\end{theorem}

\begin{proof}
Suppose $\ddagger$ is solved when the ancillary variables all vanish, \textit{i.e.} $d^k = 0$ for all $k = 1,\ldots, M$. Let $\mathfrak{S} \in \mathbb{N}_{\geq 0}$ denote the total number of satisfied clauses of $\ddagger$, and let $\mathcal{V}, \mathcal{S} \in \mathbb{N}_{\geq 0}$ denote the number of violated and satisfied clauses of $\dagger$ respectively. We have the following linear Diophantine equation: 
\begin{align}
\label{eqn:d0LDE}
    \mathfrak{S} & = 6 \cdot \mathcal{V} + 7 \cdot \mathcal{S}.
\end{align}
Since $7$ is prime, this equation has \textit{integer} solutions and can be solved completely by Euclidean Algorithm. Hence, there exists $\hat{\mathcal{V}}, \hat{\mathcal{S}} \in \mathbb{Z}$ such that the \textit{integer} solutions of Eq. \eqref{eqn:d0LDE} take the form: 
\begin{align}
\label{eqn:d0LDE:soln}
    \mathcal{V} = \hat{\mathcal{V}} - 7 \cdot n; & \quad \mathcal{S} = \hat{\mathcal{S}} + 6 \cdot n, \quad \forall n \in \mathbb{Z}.
\end{align}
Equivalently and more explicitly, since $\mathcal{V} + \mathcal{S} = M$ by construction, $\mathcal{V},\mathcal{S} \in \mathbb{N}_{\geq 0}$ are the \textit{non-negative integer} solutions of the following linear system, whose integrality and existence is due to Eq. \eqref{eqn:d0LDE:soln}:
\begin{align*}
    \begin{bmatrix}
        6 & 7 \\
        1 & 1
    \end{bmatrix} \cdot 
    \begin{bmatrix}
        \mathcal{V} \\
        \mathcal{S}
    \end{bmatrix} =
    \begin{bmatrix}
        \mathfrak{S} \\
        M
    \end{bmatrix} .
\end{align*}
\end{proof}

In the special case where the converted Max $2$-SAT problem is solved and all ancillary variables are $1$, the number of violated and satisfied clauses in the original $3$-SAT problem can be determined in two distinct ways. The first approach uses a similar argument as the proof of Theorem~\ref{thm:retrieve:d0} and employs a combinatorial argument using a linear Diophantine equation, which leads to an under-determined linear system. Since the number of satisfied and violated clauses are bounded integers, it is sufficient to enumerate the finite set of possible cases to determine the number of satisfied and violated clauses in the original $3$-SAT problem. To obtain an exact solution through a combinatorial method, we introduce additional information by assuming that we can count the total number of satisfied clauses in the converted Max $2$-SAT problem from Case 1 of Table~\ref{table.sat_vio_table}. With this additional information, the number of satisfied and violated clauses in the original $3$-SAT problem can be fully retrieved in a combinatorial manner. The above argument is summarize as follows:

\begin{corollary}
\label{thm:retrieve:d1}
Let $\ddagger := \bigwedge_{k=1}^{M} \Omega'_k$ denote a Max 2-SAT problem converted from a 3-SAT problem $\dagger := \bigwedge_{k=1}^M \Omega_k$. Suppose $\ddagger$ is solved with all ancillary variables equal to 1, and the total number of satisfied clauses in $\ddagger$ is known. By enumeration, the number of violated and satisfied clauses in $\dagger$ can be determined. Furthermore, if the number of satisfied clauses corresponding to Case 1 of $\ddagger$ is also known, then the exact number of violated and satisfied clauses in $\dagger$ can be retrieved.
\end{corollary}

\begin{proof}
    See Appendix \ref{app.pf.cor}.
\end{proof}

Notice that in the general case the clauses of the converted Max $2$-SAT problem $\ddagger$ can be partitioned into two distinctive parts, where the ancillary is either $0$ or $1$.\footnote{In practice, this can be directly counted from the final solution.} We can, therefore, apply Theorem \ref{thm:retrieve:d0} and Corollary \ref{thm:retrieve:d1} separately to both parts and recombine the results. In particular, we now show that once $\ddagger$ is solved, the number of violated and satisfied clauses in the original 3-SAT problem can either be determined or retrieved under the conditions stated in Corollary~\ref{thm:retrieve:d1}.

\begin{theorem}
\label{thm:retrieve:dall}
Let $\ddagger := \bigwedge_{k=1}^{M} \Omega'_k$ denote a Max $2$-SAT problem converted from $3$-SAT problem $ \dagger := \bigwedge_{k=1}^M \Omega_k$. Suppose $\ddagger$ is solved and the total number of satisfied clauses of $\ddagger$ is counted for both cases when the ancillary variables is $0$ or $1$. By enumeration, the number of violated and satisfied clauses in $\dagger$ can be determined. Furthermore, if the number of satisfied clauses corresponding to Case 1 of $\ddagger$ is also known, then the exact number of violated and satisfied clauses in $\dagger$ can be retrieved.
\end{theorem}

\begin{proof}
    See Appendix \ref{app.pf.thmall}.
\end{proof}

\section{Numerical Experiments}
\label{sec:experiment}
We now compare the performance of QUBO solvers \cite{DunningEtAl2018} (including Ising solvers \cite{goto}) against Relaxable Cardinality Constraints (RC$2$) \cite{rc2} --- CDCL-based state-of-the-art Max SAT solver.\footnote{RC$2$ won the \textit{weighted} and \textit{unweighted} categories of the Max SAT Evaluation competition in $2018$ \cite{competition2018} and $2019$ \cite{competition2019}.} Specifically, we focus on the accuracy, as measured by the number of violated clauses. The QUBO (and Ising) solvers use the gadget-based \cite{GAREY1976237} conversion methods outlined in Section \ref{sec:sec2} (and Appendix \ref{appendixb}), whereas RC$2$ solves the original $3$-SAT problem.

In Section \ref{sec:sub1} we used a set of four publicly available benchmark instances with clause density in the range of $1.6 \leq \rho \leq 6$, where $\rho = M/N$ denote the clause density \cite{hard1,hard2} with $M$ denoting the number of clauses and $N$ denoting the number of literals. In Section \ref{sec:sub2}, on the other hand, we used randomly generated instances (see Section \ref{sec.code} for details) and simulated two distinct scenarios. First, a \textit{generic} scenario with varying clause density from $\rho=0.5$ to $6$ (see Fig. \ref{fig:figure2}), then followed by a simulation in the \textit{hard} regions, as defined in the \cite{mezard, mezard2} (see Fig. \ref{fig:figure3}.).

For the QUBO solvers, we evaluated all $27$ heuristics listed in \cite{DunningEtAl2018} and identified the ``best" performing one, in this case the genetic algorithm with local search heuristics. Here, ``best" is defined as the heuristic that produces the smallest median value. In instances where several heuristics yield identical median values, the one with the smallest minimum value is selected. For the Ising solver, we implemented ballistic simulated bifurcation (bSB) dynamics, utilizing the same parameter settings described in \cite{goto}, but extending the number of time steps $\Delta_{t}$ to $20$ and maintaining a fixed $5000$ iterations. For performance comparison, we implemented the same selection strategy as outlined above. For RC$2$, default parameters were employed. All results were generated in Ubuntu $20.04.6$ LTS on an Intel Core i$7$-$1185$G$7$ $3.0$GHz processor with $16$GB of RAM. 

\subsection{Benchmark Comparison}
\label{sec:sub1}
We tested four sets of benchmark $3$-SAT instances from the Satisfiability Library \cite{benchmark} --- uf50-218, AIM, PRET, and DUBIOS.
\begin{itemize}
    \item uf50-218 --- we randomly selected $200$ instances out of $1000$ \textit{satisfiable} ones in total. All of the tested instances are of size $N=50$ variables and $M=218$ clauses, corresponding to $\rho=4.4$.
    \item AIM --- we tested all $16$ \textit{satisfiable} instances of size $N=50$ variables, where the number of clauses varies from $M=80$ to $300$; corresponding to clause density from $\rho=1.6$ to $6$.
    \item PRET --- we tested all $4$ \textit{unsatisfiable} instances of size $N=60$ variables, $M=160$ clauses, and $\rho = 2.7$.
    \item DUBIOS --- we tested $7$ \textit{unsatisfiable} instances with varying sizes, from $N=60$ to $78$ variables and from $M=160$ to $208$ clauses (for all instances $\rho = 2.7$).
    
\end{itemize}
The comparison of the performance of QUBO, Ising, and RC$2$ solvers across the four sets of instances is shown in Fig. \ref{fig:combined}.

\begin{figure}[!ht]
\centering
\includegraphics[width=\linewidth]{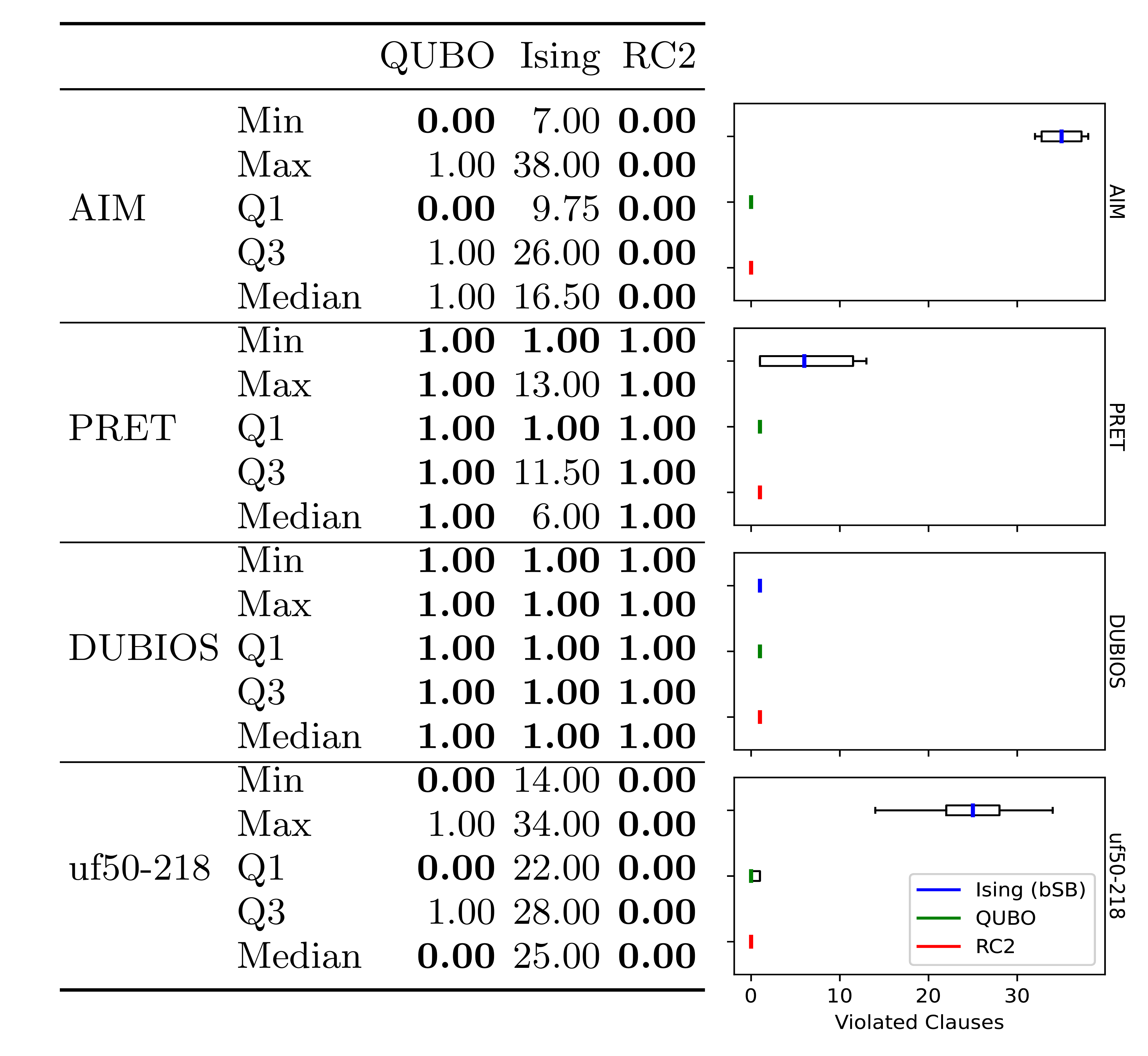}
\caption{\textbf{Benchmark.} On the left side is a summary of the statistics of the performed numerical simulations. The highlighted values represent the best solutions for each of the four sets. The Q1, Q3, and Median values were calculated using the \texttt{percentile} function from \texttt{NumPy} (in \texttt{Python}). On the right side are the boxplots, corresponding to the instance with largest $N$ from each set. In case two instances from the same set have equal $N$, we selected the one with higher clause density $\rho$.}
\label{fig:combined}
\end{figure}

Evidently, QUBO matches the accuracy of RC$2$ for almost all of the tested instances. For the Ising solver, however, we observe that it predominantly produces trivial solutions, such as all unnegated or negated variables, even after parameter tuning. Nevertheless, in certain instances from the PRET and DUBIOS sets, Ising achieves accuracy comparable to that of QUBO and RC$2$.  {Yet}, this only occurs because, in these specific cases, the trivial solutions coincide with the actual solutions.

\subsection{Clause Density}
\label{sec:sub2}
In practice for many combinatorial optimization problems, large performance discrepancies occur when problem parameters are changed \cite{hard1,hard2}. Specifically for random $3$-SAT, where the ratio between the number of clauses and literals is known to determine the hardness of the problem, this is reflected by the observation that varying the number of clauses induces phase transitions \cite{pahsetransition}; where for $\rho>\rho_{c}$ (with $\rho_{c}=4.26$ being a threshold, as computed in \cite{mezard,mezard2}) generic problems are no longer SAT \cite{hard3}. To that end, in this Section we  {provide} additional insight into how the clause density $\rho$ affects the performance of QUBO solvers and RC$2$. 

 {We begin with the generic scenario.} The simulations are conducted on randomly-generated instances of sizes $N=30,50,60,70$ variables. Here, each clause is constructed from exactly three distinct variables, picked uniformly at random, where each variable has $1/2$ probability of being negated. The clause density for each $N$ varies from $\rho=0.5$ to $6$ with a step size of $0.5$. For each $N$, we generate $10$ instances for every value of $\rho$. The instances are chosen so that we can finish the numerical simulations within a reasonable time. Indeed, as we can observe from Fig. \ref{fig:figure2}, as $\rho$ increases in the given range, the problem exhibits a typical easy-to-hard pattern for all tested solvers, affecting their performance.

\begin{figure}[!ht]
\centering
\includegraphics[width=\linewidth]{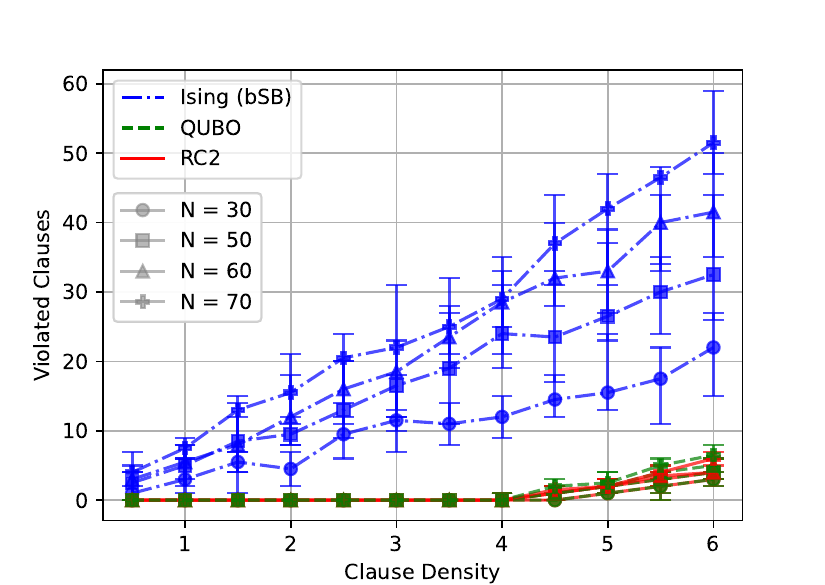}
\caption{The number of violated clauses as a function of the clause density for four different instance sizes $N$. Each datapoint corresponds to 10 randomly-generated instances with 3 trials per instance.}
\label{fig:figure2}
\end{figure}

In alignment with our previous results, we observe that, even for the instances  {with} high clause density $\rho$, QUBO consistently matches the accuracy of RC2 across all sizes. What makes the results particularly appealing is the fact that popular alternative approaches, \textit{e.g.} quantum approximate optimization algorithm (QAOA) \cite{qaoa} suffer from so-called \textit{reachability deficit} (see, \textit{e.g.}, \cite{deficit}) when dealing with $3$-SAT problems with high clause density $\rho$. Typically, this is attributed to the amount of frustration in the Ising variables in the cost Hamiltonian that encodes the problem. Consequently, QAOA has significant hardware overhead (especially in the hard region of the problem where the easy-to-hard pattern emerges), making it unsuitable for current NISQ devices. Despite continued attempts \cite{qaoa1,qaoa2,qaoa3} to improve the capacity of QAOA to handle hard 3-SAT instances on near-term devices, its issues remain; further motivating the use of available QUBO solvers in the current hardware landscape.

The Ising solver, on the other hand, similar to subsection \ref{sec:sub1}, cannot solve the randomly generated instances and produces mostly trivial solutions, \textit{i.e.} all variables are either unnegated or negated. Unlike the results shown in Fig. \ref{fig:combined}, however, here the trivial solutions do not correspond to the actual solutions, and so the number of violated clauses is strictly bigger than zero even for $\rho=0.5$. Nonetheless, its trivial solutions do, in fact, satisfy \textit{some} of the clauses, thus reducing the overall number of the violated clauses.

{We now turn to the scenario in the \textit{hard} regime. Besides the $\rho_{c}=4.26$ threshold, the authors of \cite{mezard,mezard2} also discovered the existence of a particular region $3.92<\rho<\rho_{c}$. In this particular region, the problems are notoriously difficult to solve due to an increase in the number of metastable states (for a thorough examination of this phenomena in spin glass models, see \cite{spinglass}), even though they remain SAT. The simulations are conducted on randomly-generated instances of $N=70$ variables, where similar to the generic case, each clause is constructed from exactly three distinct variables (picked uniformly at random), with each having $1/2$ probability of being negated. The clause density is now restricted to the aforementioned region $3.92<\rho<\rho_{c}$, where the number of clauses increases from $260$ to $310$.\footnote{Notice that we also included small neighborhood before and after the mentioned region.} Likewise, $10$ instances are generated for every value of $\rho$.}

\begin{figure}[!ht]
\centering
\includegraphics[width=\linewidth]{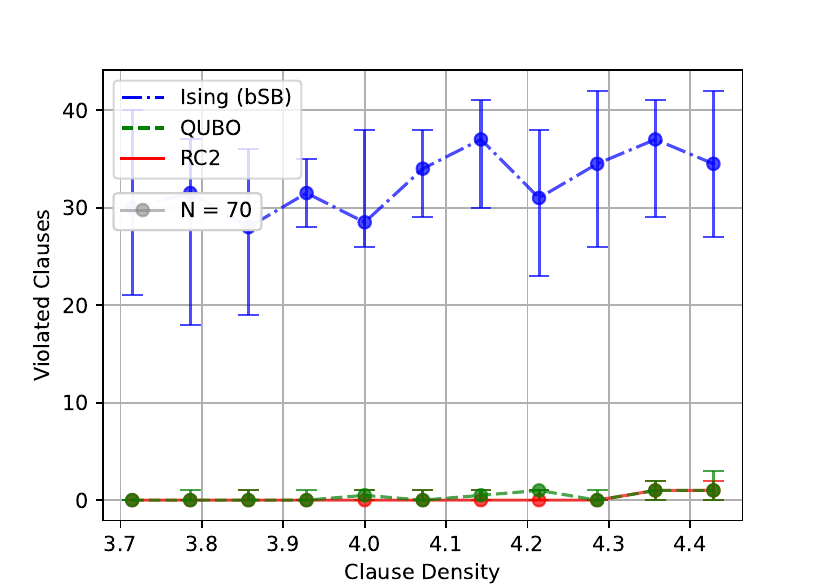}
\caption{The number of violated clauses as a function of the clause density for $N=70$, restricted to the \textit{hard} intermediate region, derived in \cite{mezard,mezard2}. Each datapoint corresponds to 10 randomly-generated instances with 3 trials per instance.}
\label{fig:figure3}
\end{figure}
 {In agreement with the previous results, we can see from Fig. \ref{fig:figure3} that even in this particular region the QUBO solver matches the accuracy of RC2, while the Ising solver still outputs trivial solutions, where all variables are unnegated or negated.}

\section{Conclusions}
This work paves the way towards establishing QUBO solvers as an alternative platform to CDCL-based solvers for tackling $3$-SAT problems, and it is the first performance-oriented comparison between them. To that end, we first demonstrated a $2$-step technique (using the $(7,10)$-gadget \cite{GAREY1976237}) for mapping $3$-SAT problems of arbitrary size to the QUBO form via an intermediate step of converting them to their corresponding Max $2$-SAT formulations, and then deriving the corresponding QUBO cost function. The conversion is appealing as it allows us to retrieve the solution of the original $3$-SAT problem from the QUBO solution. This is necessary to enable the use of current QUBO (and Ising) solvers for $3$-SAT problems, see Appendix \ref{appendix} and \ref{appendixb}.

We showed rigorously that the number of both satisfied and violated clauses of the original hard $3$-SAT problem can be computed from the converted Max $2$-SAT problem. Through numerical simulations on several public benchmark instances, and randomly-generated instances within the well-known \textit{hard} region \cite{mezard, mezard2}, we demonstrated that available QUBO solvers can achieve state-of-the-art accuracy, matching that of RC$2$. Besides, in the current NISQ era, the results have direct utility with regard to solving $3$-SAT problems on available quantum annealing hardware, and facilitate the development of advanced digital QUBO solvers and heuristics.

\section{Data Availability}
\label{sec.code}
The entire source code and all of the raw data for the numerical analysis are openly available at \href{https://github.com/seashell-s/qubo-3sat}{https://github.com/seashell-s/qubo-3sat}.

\appendix
\section{3-SAT Hamiltonian does not have Ising quantization}
\label{appendix}

Here we show that the 3-SAT Hamiltonian does not have an Ising quantization, necessitating the conversion approaches, outlined in Section \ref{sec:sec2} and Appendix \ref{appendixb}, to utilize QUBO and Ising solvers for tackling 3-SAT problems.\\
To ease the notation, for this Section we omit the upper index of the clause and let $S := \left\{\sigma_i \right\}_{i=1}^N$ denote the Ising-spin literals of a 3-SAT problem with $M$ clauses $\left\{C_\mu\right\}_{\mu=1}^{M}$. That is for the following discussion we have $\sigma_i \in \left\{\pm 1\right\}$ and Eq. 2 of \cite{barthel} reads:

\begin{align}
\label{eqn.barthel}
    \mathcal{H}=\frac{M}{8} - \sum_{i=1}^N H_i \sigma_i - \sum_{i<j}T_{ij}\sigma_i \sigma_j - \sum_{i<j<k}J_{ijk} \sigma_i \sigma_j \sigma_k , 
\end{align}
where $c_{\mu,i} = 1$ if $\sigma_i$ is in clause $C_\mu$ and $c_{\mu,i} = -1$ otherwise. $H_i = \frac{1}{8}\sum_{\mu} c_{\mu,i}, T_{ij} = -\frac{1}{8}\sum_{\mu}c_{\mu,i}\cdot c_{\mu,j}, J_{ijk} = \frac{1}{8}\sum_{\mu}c_{\mu,i}\cdot c_{\mu,j}\cdot c_{\mu,k}$
\\

\begin{theorem}
\label{thm.noquad}
Quadratic penalty formulation of Eq. \eqref{eqn.barthel} is not possible for Ising-spin literals $\sigma_i \in \left\{\pm 1\right\}$. 
\end{theorem}
\begin{proof}
We begin by labeling the components of Eq. \eqref{eqn.barthel} as follows:
\begin{align*}
    \mathcal{H}= \frac{M}{8} - \underbrace{\sum_{i=1}^N H_i \sigma_i}_{\mathcal{L}} - \underbrace{\sum_{i<j}T_{ij}\sigma_i \sigma_j}_{\mathcal{Q}} - \underbrace{\sum_{i<j<k}J_{ijk} \sigma_i \sigma_j \sigma_k}_{\mathcal{C}} .
\end{align*}

The first three terms of $\mathcal{H}$ are at most quadratic, therefore it suffices to show that the cubic term $\mathcal{C}$ does not admit a quadratization over $\pm1$ Ising spins. Suppose by contrary that $\mathcal{H}$ (equivalently $\mathcal{C}$) can be expressed as a quadratic polynomial over $\pm1$ Ising spins with, at most, a quadratic penalty term. That is, for each $j,k = 1,\ldots, N$, we can define an ancillary variable  $y_{jk} \in \left\{-1,1\right\}$, a constant $\mu_k \in \mathbb{R}_{\geq 0}$, and a function $p_k: S\times S\times \left\{\pm 1\right\} \rightarrow \mathbb{R}$ such that:
\begin{enumerate}[label=\textbf{P.\arabic*}]
        \item \label{p.1} $p_k$ is non-negative polynomial with $\operatorname{deg}(p_k) \leq 2$
        \item \label{p.2} $p_k$ vanishes \textbf{only} when $\sigma_j \sigma_k - y_{jk} = 0$ for each $k= 1,\ldots,N$. 
\end{enumerate}
Eq. \eqref{eqn.barthel} thus becomes: 
\begin{align}
    \mathcal{H}= \frac{M}{8} &- \mathcal{L} - \mathcal{Q}  \\ \nonumber
    & - \left( \sum_{i<j<k}J_{ijk} \sigma_i y_{jk} +
    \underbrace{\sum_{k=1}^N \mu_k p_k \left( \sigma_j, \sigma_k, y_{jk} \right)}_{\operatorname{Penalty}}\right).
\end{align}
By \ref{p.1}, $p_k$ is either linear or quadratic. The only linear polynomial (over the ancillary variable) that satisfies \ref{p.2} is $\sigma_j \sigma_k = y_{jk}$, which violates the non-negativity condition in \ref{p.1} when $\left(\sigma_j, \sigma_k, y_{jk}\right) = \left(1,-1,1\right)$. Therefore $p_k$ must be quadratic over all variables and there exists constants $A,B,C,D,E,F,G,H,I,J \in \mathbb{R}$ such that $p_k$ can be expanded into the following form:
\begin{align}
\label{eqn.quadf}
    p_k &= B\sigma_j + C\sigma_k + Dy_{jk} \\
    & \quad + A\sigma_j\sigma_k + E\sigma_jy_{jk} + F\sigma_ky_{jk} \nonumber \\ 
    & \quad + Gy_{jk}^2 + H\sigma_j^2 + I\sigma_k^2 + J  . \nonumber
\end{align}
In Table \ref{table.quad}, we summarize the possible combinations of $\left(\sigma_j, \sigma_k, y_{jk}\right)$ and the corresponding coefficients of $f$ such that both \ref{p.1} and \ref{p.2} are satisfied. 

\begin{table}[ht!]
\begin{center}
\begin{tabular}{|c|c|c|c|}
\hline
$\sigma_j$ & $\sigma_k$ & $y_{jk}$ & Constants when $p_k = 0$:                 \\ \hline
$1$     & $1$     & $1$        & $J = -A - B - C - D - E - F - G - H - I$ \\ \hline
$-1$    & $1$     & $-1$       & $J = A + B - C + D + E + F - H - H - I$  \\ \hline
$1$     & $-1$    & $1$        & $J = A - B + C + D + E - F - G - H - I$  \\ \hline
$-1$    & $-1$    & $-1$       & $J = -A + B + C - D - E + F - G - H - I$ \\ \hline
\end{tabular}
\end{center}
\caption{Choices of $\left(\sigma_j, \sigma_k, y_{jk}\right)$ and coefficients of $p_k$, such that both \ref{p.1} and \ref{p.2} are satisfied.}
\label{table.quad}
\end{table}
By \ref{p.2}, we must have that all four rows of the right-most column of Table \ref{table.quad} are satisfied simultaneously. Hence $ A = B= C= D= E= F = 0$, and Eq. \eqref{eqn.quadf} reduces to: 
\begin{align}
\label{eqn.reduced.quad}
    p_k  = Gy_{jk}^2 + H\sigma_j^2 + I\sigma_k^2 + J; \quad J = -G - H - I,
\end{align}
which vanishes for all $\left(\sigma_j, \sigma_k, y_{jk}\right) \in S \times S \times \left\{\pm 1\right\}$ and \ref{p.2} is violated. 

To see this, consider arbitrary Ising spin configuration $\left(\sigma_j, \sigma_k, y_{jk}\right) \in S \times S \times \left\{\pm 1\right\}$, Eq. \eqref{eqn.reduced.quad} then becomes:
\begin{align*}
    p_k\left(\sigma_j, \sigma_k, y_{jk}\right) & = G\cdot 1 + H\cdot 1 + I\cdot 1  + J \\ 
    & = G + H + I - G - H - I = 0.    
\end{align*}

This implies any quadratic penalty reduces to Eq. \eqref{eqn.reduced.quad} necessarily violates  \ref{p.2}. We therefore conclude that a quadratic penalty $p_k$ that satisfies both \ref{p.1} and \ref{p.2} cannot exist, contradicting our claim that $\mathcal{H}$ can be expressed at most quadratically  with such a penalty function and the proof is complete. 
\end{proof}

By Theorem \ref{thm.noquad}, it follows immediately that a quadratic Hamiltonian formulation of the 3-SAT problem (\textit{e.g.} as suggested by Eq. 4 of \cite{chancellor}) is \textbf{not possible} in the Ising spin variables $\sigma_i^z \in \left\{\pm 1 \right\}$. The referenced equation is included below for completeness with $N$ replacing $k$ as the end point of the summation, denoting the number of variables, to avoid confusion:
\begin{align*}
    \mathcal{H}_{\text{clause}}^{(2)} &= \underbrace{J\cdot \sum_{j<i}^N c(i) c(j) \sigma_i^z \sigma_j^z}_{\text{Quadratic term}} + \underbrace{h\cdot \sum_{i=1}^N c(i) \sigma_i^z }_{\text{Linear term}} \\
    &+ \underbrace{J^a \sum_{i=1}^N \sum_{j=1}^N c(i) \sigma_i^z \sigma_{j,a}^z}_{\text{Quad. coupled with aux from Cubic}} + \underbrace{\sum_{i=1}^N h_i^a \sigma_{i,a}^z}_{\text{aux. from Cubic}}
\end{align*}

\section{Max 2-SAT to Ising}
\label{appendixb}

Similar to section \ref{sec:2toqubo}, we must construct a problem Hamiltonian $H_{\ddagger}$ \cite{map} in order to utilize Ising solvers for Max $2$-SAT problems. Consider a Max $2$-SAT problem $\ddagger := \bigwedge_{k=1}^{M'} \Omega'_k$ with $M'$ clauses $\left\{ \Omega'_k\right\}_{k=1}^{M'} $and $N'$ literals $\left\{\sigma'_i\right\}_{i=1}^{N'}$. We now define a $H_\ddagger$ on $\ddagger$ over the search space $\mathbf{x} \in \left\{0,1\right\}^{N'}$, for which we minimize over with the goal to minimize the total number of violated clauses. We first consider the negation binary variable $\mathbf{y} \in \left\{0,1\right\}^{N'}$ of $\mathbf{x}$ as follows:
\begin{align*}
    y_j = 1 - x_j,\quad \operatorname{for} j = 1,\ldots,N'.
\end{align*}
We then map binary variables $\{y_{j}\}_{j=1}^{N'}$ to the $\pm1$ eigenstates $\{\overline{\sigma}_{j}^{k}\}_{j=1}^{N'}$ of the Pauli-$z$ operator:
\begin{equation}
\overline{\sigma}_{j}^{k} \ket{y_{j}} = (-1)^{y_{j}}\ket{y_{j}} .
\end{equation}

We then define the set of intermediate variables  $\{v_{j}^{k}\}$ over the $M'$ clauses of $\ddagger$ as in section \ref{sec:2toqubo}.

Now we are ready to translate each $2$-literal clause $\Omega_{k}'$ from the Max $2$-SAT problem $\ddagger$ into its corresponding $2$-local term $H_{\Omega_{k}'}$

\begin{equation}
H_{\Omega_{k}'} = \frac{1-v_{j_{1}}^{k}\overline{\sigma}{j_{1}}^{k}}{2}\frac{1-v_{j_{2}}^{k}\overline{\sigma}_{j_{2}}^{k}}{2}
\end{equation}

The problem Hamiltonian $H_{\ddagger}$ can then be constructed in the following compact form

\begin{equation}
\label{eq:hamiltonian}
H_{\ddagger} = \sum_{\Omega_{k}'} H_{\Omega_{k}'}
\end{equation}
where we simply sum over the energies of all $M'$ clauses in the Max $2$-SAT problem.\footnote{Interestingly, by virtue of the Hamiltonian (Eq. \eqref{eq:hamiltonian}) being the sum over all the energies of the $M'$ clauses, its ground state corresponds to the solution of the Max $2$-SAT problem. Namely, the ground state could have positive energy, corresponding to the minimum number of violated clauses.} This way, the converted Max $2$-SAT problem can be formulated in terms of the \textit{computational energy}, where identical energy penalties of $1$ are assigned for every violated clause $H_{\Omega_{k}'}$, and $0$ for all satisfied ones. Expanding Eq. \eqref{eq:hamiltonian} then gives\footnote{Notice the similarity with Eq. \eqref{eq:qubo}.}:

\begin{equation}
H_{\ddagger} = \frac{1}{4}\sum_{\Omega_{k}'}1-v_{j_{1}}^{k}\overline{\sigma}^{k}_{j_{1}}-v_{j_{2}}^{k}\overline{\sigma}^{k}_{j_{2}}+v_{j_{1}}^{k}v_{j_{2}}^{k}\overline{\sigma}^{k}_{j_{1}}\overline{\sigma}^{k}_{j_{2}}
\end{equation}
where after rescaling by a factor of $4$ and dropping the first constant term we obtain, for each of the literals indexed by $j_i \in \left[1,\ldots,N'\right]$, the entries of the external field $\left[h_{j_i}\right] \in \mathbb{R}^{N'}$ and the coupling term $\left[J_{j_1j_2}\right] \in \mathbb{R}^{\left(N'\right) \times \left(N'\right)}$

\begin{equation}
J_{j_{1}j_{2}} =  \sum_{k = 1}^{M'}v_{j_{1}}^{k}v_{j_{2}}^{k}
\end{equation}

\begin{equation}
h_{j_{i}} = - \sum_{k= 1}^{M'} v_{j_{i}}^{k}
\end{equation}

Notice that each of the entries in the Max $2$-SAT formulation, defined by the indices of the literals, is a sum over the set of $M'$ clauses.

\section{Proof of Cor. \ref{thm:retrieve:d1}}
\label{app.pf.cor}
\begin{proof}

Suppose $\ddagger$ is solved when the ancillary variables satisfy $d^k =1$ for all $k = 1,\ldots, M$. Let $\mathfrak{S} \in \mathbb{N}_{\geq 0}$ denote the total number of satisfied clauses of $\ddagger$, and let $\mathcal{V}, \mathcal{S} \in \mathbb{N}_{\geq 0}$ denote the number of violated and satisfied clauses of $\dagger$ respectively. We have the following linear Diophantine equation of \textit{three} variables:
    \begin{align}
    \label{eqn:d1LDE}
        \mathfrak{S} & = 6 \cdot \mathcal{S}_{1,2} + 7 \cdot \mathcal{S}_{1,3} + 4 \cdot \mathcal{V} \\
        \mathcal{S} & = \mathcal{S}_{1,2} + \mathcal{S}_{1,3} \nonumber ,
    \end{align}
    where $\mathcal{S}_{1,2}, \mathcal{S}_{1,3} \in \mathbb{N}_{\geq 0}$ denote the total number of satisfied clauses under with ancillary variable $1$ in Case $2$ and Case $3$ in Table \ref{table.sat_vio_table} respectively. 
    By construction, since $\mathcal{V} + \mathcal{S} = M$ (where $M$ denote the total number of clause in $\dagger$), we arrive at the following set of \textit{under-determined} linear system:
    \begin{align*}
        \mathfrak{S} & = 6 \cdot \mathcal{S}_{1,2} + 7 \cdot \mathcal{S}_{1,3} + 4 \cdot \mathcal{V} \\
        \mathcal{S} & = \mathcal{S}_{1,2} + \mathcal{S}_{1,3} \\
        M & = \mathcal{V} + \mathcal{S}.
    \end{align*}
    Since all of variables $\mathcal{S}_{1,2}, \mathcal{S}_{1,3}, \mathcal{V}$ are non-negative integers and $\gcd(4,6,7) = 1$, this system admits finitely many integer solutions. These solutions can be enumerated to determine the number of satisfied and violated clauses in the original 3-SAT problem $\dagger$.

    Now, suppose the exact number of clauses of \textit{Case 1} (i.e., clauses with the $(0,0,0)$ assignment) can be directly determined from the solved Max $2$-SAT problem $\ddagger$. In this case, $\mathcal{V}$ (the number of violated clauses in $\dagger$) becomes known. Once $\mathcal{V}$ is known, $\mathcal{S} = \mathcal{S}_{1,2} + \mathcal{S}_{1,3}$ (the number of satisfied clauses in $\dagger$) can be retrieved by solving the following linear system:
    \begin{align*}
    \begin{bmatrix}
        6 & 7 \\
        1 & 1
    \end{bmatrix} \cdot 
    \begin{bmatrix}
        \mathcal{S}_{1,2} \\
        \mathcal{S}_{1,3}
    \end{bmatrix} =
    \begin{bmatrix}
        \mathfrak{S} - 4\mathcal{V} \\
        M - \mathcal{V}
    \end{bmatrix} .
\end{align*}

This system is now fully determined, and solving it gives the exact integer values of $\mathcal{S}_{1,2}$ and $\mathcal{S}_{1,3}$, from which the number of satisfied clauses $\mathcal{S} = \mathcal{S}_{1,2} + \mathcal{S}_{1,3}$ in the original $3$-SAT problem $\dagger$ can be deduced.
\end{proof}

\section{Proof of Theorem~\ref{thm:retrieve:dall}}
\label{app.pf.thmall}

\begin{proof}
    Let $I_a := \left\{k \middle| d^k = a \right\} \subset \left\{1,\ldots, M\right\}$, then $I_0 \sqcup  I_1 = \left\{1,\ldots, M\right\}$. Consider the partition of the set of converted clauses $\left\{ \Omega'_k\right\}_{k=1}^{M}$ of $\ddagger$ by the sets of indices $I_0, I_1$: Let $\Omega'^a := \left\{ \Omega'_k \right\}_{k\in I_a}$ for $a = 0,1$, then $\left\{ \Omega'_k\right\}_{k=1}^{M} = \Omega'^0 \sqcup \Omega'^1$. 

    Suppose $\ddagger$ is solved and let $\mathfrak{S}_a \in \mathbb{N}_{\geq 0}$ denote the total number of satisfied clauses of $\Omega'^a \subset \left\{ \Omega'_k\right\}_{k=1}^{M}$ for $a = 0,1$ respectively, and let $\mathcal{V}_a, \mathcal{S}_a \in \mathbb{N}_{\geq 0}$ denote the corresponding number of violated and satisfied clauses of $\dagger$ for $a = 0,1$ respectively.

    In the case when the ancillary variables are $1$, \textit{i.e.} in $\Omega'^1 \subset \left\{ \Omega'_k\right\}_{k=1}^{M}$, let $\mathcal{S}_{1,2}, \mathcal{S}_{1,3} \in \mathbb{N}_{\geq 0}$ denote the total number of satisfied clauses with ancillary variable $d^k = 1$ under Case $2$ and Case $3$, respectively, in Table \ref{table.sat_vio_table}. Then $\mathcal{V}_1, \mathcal{S}_1 \in \mathbb{N}_{\geq 0}$ are integer solutions of the linear Diophantine equation:
    \begin{align*}
        \mathfrak{S}_1 & = 6 \cdot \mathcal{S}_{1,2} + 7 \cdot \mathcal{S}_{1,3} + 4 \cdot \mathcal{V}_1 \\
        \mathcal{S}_1 & = \mathcal{S}_{1,2} + \mathcal{S}_{1,3} ,
    \end{align*}
    which, by Corollary \ref{thm:retrieve:d1}, we can either:
    \begin{enumerate}
        \item Determine $\mathcal{S}_{1,2}, \mathcal{S}_{1,3}, \mathcal{V}$ by enumeration, or
        \item Retrieve $\mathcal{S}_{1,2}, \mathcal{S}_{1,3}, \mathcal{V}$ by counting the the exact number of clauses in Case $1$ of Table~\ref{table.sat_vio_table}.
    \end{enumerate}

    For the subset of clauses where the ancillaries are $0$, \textit{i.e.} in $\Omega'^0 \subset \left\{ \Omega'_k\right\}_{k=1}^{M}$, $\mathcal{V}_0, \mathcal{S}_0 \in \mathbb{N}_{\geq 0}$ are integer solutions of the linear Diophantine equation:
    \begin{align*}
        \mathfrak{S}_0 = 6 \cdot \mathcal{V}_0 + 7 \cdot \mathcal{S}_0,
    \end{align*}
    Eq. \eqref{eqn:d0LDE} of Theorem \ref{thm:retrieve:d0} can be written in terms of Case $2$ and Case $3$ of Table \ref{table.sat_vio_table}: Let $\mathcal{S}_{0,2}, \mathcal{S}_{0,3} \in \mathbb{N}_{\geq 0}$ denote the total number of satisfied clauses with ancillary variable $d^k = 0$ under Case $2$ and Case $3$, respectively, in Table \ref{table.sat_vio_table}, then by construction $\mathcal{S}_0 = \mathcal{S}_{0,2} + \mathcal{S}_{0,3}$ and:
    \begin{align*}
        \mathfrak{S}_0 & = 6 \cdot \mathcal{V}_0 + 7 \cdot \mathcal{S}_0 \\
                       & = 6 \cdot \mathcal{V}_0 + 7 \cdot \mathcal{S}_{0,2} + 7 \cdot \mathcal{S}_{0,3}.
    \end{align*}
    which can be determined explicitly by Theorem \ref{thm:retrieve:d0}.

    The total number of violated and satisfied clauses of $\dagger$ is thus given by $\mathcal{V} = \mathcal{V}_0 + \mathcal{V}_1$ and $\mathcal{S} = \mathcal{S}_0 + \mathcal{S}_1$, respectively, as desired.

    \begin{table}[ht!]
        \begin{center}
        \begin{tabular}{|c|c|c|} \hline
                   & $d= 0$, Thm  \ref{thm:retrieve:d0}  & $d = 1$, Cor  \ref{thm:retrieve:d1} \\ \hline
            Case 1 & $6\cdot \underbrace{S_{0,1}}_{\mathcal{V}_0}$ & $4 \cdot \underbrace{S_{1,1}}_{\mathcal{V}_1}$ \\ \hline
            Case 2 &  $7 \cdot \mathcal{S}_{0,2}$ & $6 \cdot \mathcal{S}_{1,2}$ \\ \hline
            Case 3 &  $7 \cdot \mathcal{S}_{0,3}$ & $7 \cdot \mathcal{S}_{1,3}$ \\ \hline
                   &  $+= \mathfrak{S}_0$         & $+= \mathfrak{S}_1$ \\ \hline
        \end{tabular}
        \end{center}
    \caption{Case by case satisfied clauses, notice the $d^k = 0,1$ case are disjoint with the total number of satisfied clauses of $\dagger$ in each part being $\mathcal{S}_0 = \mathcal{S}_{0,2} + \mathcal{S}_{0,3}, \mathcal{S}_1 = \mathcal{S}_{1,2} + \mathcal{S}_{1,3}$.}
    \label{table.extraproof}
    \end{table}

    To see that the sets are disjoint and, therefore, none of the pieces are being double counted, we enumerate them in Table~\ref{table.extraproof} above. 
    \end{proof}

\end{document}